\documentclass[a4paper,10pt]{article}

\usepackage[utf8]{inputenc}
\usepackage{amsfonts}
\usepackage{amsmath}
\usepackage{amsthm}
\usepackage{color}

\usepackage{listings}
\usepackage{hyperref}
\usepackage{float}
\usepackage{algorithm}
\usepackage{algorithmic}

\newfloat{function}{thb}{lop}
\floatname{function}{Function}

\newcommand{\C}{\mathbb C}
\DeclareMathOperator{\diag}{diag}
\DeclareMathOperator{\argmax}{arg max}
\newcommand{\norm}[1]{\left\Vert #1 \right\Vert}

\newtheorem{theorem}{Theorem}
\newtheorem{lemma}[theorem]{Lemma}

\newcommand{\twotwo}[4]{\begin{bmatrix} #1 & #2 \\ #3 & #4\end{bmatrix}}
\newcommand{\vett}[2]{\begin{bmatrix} #1 \\ #2\end{bmatrix}}
\newcommand{\rowvett}[2]{\begin{bmatrix} #1 & #2\end{bmatrix}}

\newcommand{\abs}[1]{\left\vert#1\right\vert}

\newcommand{\matlab}{Matlab\textregistered{}}

\newcommand{\col}{:}

\DeclareMathOperator{\disp}{\nabla}

\definecolor{darkgray}{gray}{0.4}
\title{A note on the $O(n)$-storage implementation of the GKO algorithm and its adaptation to Trummer-like matrices}
\author{Federico Poloni\thanks{Scuola Normale Superiore; Piazza dei Cavalieri, 7; 56126 Pisa, Italy. E-mail \href{mailto:f.poloni@sns.it}{\texttt{f.poloni@sns.it}}}}
\date{}

\begin{document}
\maketitle
\begin{abstract}
We propose a new $O(n)$-space implementation of the GKO-Cauchy algorithm for the solution of linear systems where the coefficient matrix is Cauchy-like. Moreover, this new algorithm makes a more efficient use of the processor cache memory; for matrices of size larger than $n \approx 500-1000$, it outperforms the customary GKO algorithm.

We present an applicative case of Cauchy-like matrices with non-reconstructible main diagonal. In this special instance, the $O(n)$ space algorithms can be adapted nicely to provide an efficient implementation of basic linear algebra operations in terms of the low displacement-rank generators.
\end{abstract}

\section{Introduction}
Several classes of algorithms for the numerical solution of Toeplitz-like linear systems exist in the literature. We refer the reader to \cite{vbhk} for an extended introduction on this topic, with descriptions of each method and plenty of citations to the relevant papers, and only summarize them in the following table.

\begin{tabular}{cccp{5cm}}
Name & Operations & Memory & Stability\\
\hline
 Levinson & $O(n^2)$ & $O(n)$ & stable only for some symmetric matrices\\
 Schur-Bareiss & $O(n^2)$ & $O(n^2)$ & backward stable only for symmetric, positive definite matrices\\
 GKO & $O(n^2)$ & $O(n^2)$ & stable in practice in most cases\\
 Superfast & $O(n \log^2 n)$ & $O(n)$ & leading constant may be large; may be unstable in the nonsymmetric case
\end{tabular}
The Levinson algorithm is known to be unstable even for large classes of symmetric positive definite matrices \cite{Cybenko}; stabilization techniques such as \emph{look-ahead} may raise the computational cost from $O(n^2)$ to $O(n^3)$ or from $O(n \log^2 n)$ to $O(n^2)$. A mixed approach like the one in the classical FORTRAN code by Chan and Hansen \cite{toms729} bounds the complexity growth, but may fail to remove the instability.
The GKO algorithm is generally stabler \cite{gko95}, even though in limit cases the growth of the coefficients appearing in the Cauchy-like generators may lead to instability. Though superfast Toeplitz solvers have a lower computational cost, when the system matrix is nonsymmetric and ill-conditioned $O(n^2)$ algorithms such as the GKO algorithm \cite{gko95} are still attractive.

In this paper we will deal with the GKO algorithm. It is composed of two steps: reduction of the Toeplitz matrix to a Cauchy-like matrix with displacement rank $r=2$, which takes $O(n)$ memory locations, and $O(n \log n)$ ops, and actual solution of the Cauchy-like system via a generalized Schur algorithm, which takes $O(n^2)$ ops and $O(n^2)$ auxiliary memory locations.

In 1994, Kailath and Chun \cite{kc} showed that it is possible to express the solution of a linear system with Cauchy-like matrix as the Schur complement of a certain structured augmented matrix. In 2006, in a paper on Cauchy-like least squares problems, G.~Rodriguez exploited this idea to design a variation of GKO using only $O(n)$ memory locations.

In the first part of the present  paper, we will provide an alternative $O(n)$-space implementation of the Schur Cauchy-like system solution algorithm, having several desirable computational properties.

Moreover, in some applications, a special kind of partially reconstructible Cauchy-like matrices appear, i.e., those in which the main diagonal is not reconstructible. We shall call them Trummer-like, as they are associated with Trummer's problem \cite{gera}. We will show how the $O(n)$-storage algorithms adapt nicely to this case, allowing one to develop an integrated algorithm for their fast inversion. In particular, one of the key steps in order to obtain a full representation of their inverse is the calculation of $\diag(T^{-1})$ for a given Trummer-like $T$.

\paragraph{Structure of the paper} In \autoref{s:basic}, we will recall the concept of displacement operators, Cauchy-like and Trummer-like matrices. In sections~\ref{s:gko} and \ref{s:ar} we will study respectively the original GKO algorithm and its first $O(n)$-space variant due to Rodriguez \cite{ar}, \cite{rod}. In \autoref{s:dd} we will introduce and analyze our new $O(n)$-space variant. In \autoref{s:trummer} we will deal with system solving and matrix inversion for Trummer-like matrices. Finally, \autoref{s:experiments} is dedicated to showing some numerical experiments that confirm the effectiveness of our approach, and \autoref{s:conclusions} contains some conclusive remarks.

\section{Basic definitions}\label{s:basic}

\paragraph{Indexing and notation} We will make use of some handy matrix notations taken from FORTRAN and \matlab{}. When $M$ is a matrix, the symbol $M_{i \col j, k \col \ell}$ denotes the submatrix formed by rows $i$ to $j$ and columns $k$ to $\ell$ of $M$, including extremes. The index $i$ is a shorthand for $i\col i$, and $\col$ alone is a shorthand for $1 \col n$, where $n$ is the maximum index allowed for that row/column. A similar notation is used for vectors. When $v \in \C^n$ is a vector, the symbol $\diag(v)$ denotes the diagonal matrix $D \in \C^{n \times n}$ such that $D_{i,i}=v_i$. On the other hand, when $M\in\C^{n \times n}$ is a square matrix, $\diag(M)$ denotes the (column) vector with entries $M_{1,1},M_{2,2},\dots,M_{n,n}$.

Throughout the paper, we shall say that a vector $s \in \C^{n}$ is \emph{injective} if $s_i \neq s_j$ for all $i,j=1,2,\dots,n$ such that $i \neq j$. In the numerical experiments, we will denote by $\norm{\cdot}$ the Euclidean 2-norm for vectors and the Frobenius norm for matrices.

\paragraph{Displacement operators and Cauchy-like matrices} Let $t,s \in \C^n$. We shall denote by $\disp_{t,s}$ the operator $\C^{n \times n} \to \C^{n \times n}$ which maps $M$ to
\[
 \nabla_{t,s} (M)= \diag(t)M-M\diag(s).
\]
A matrix $C \in \C^{n\times n}$ is said \emph{Cauchy-like} (with displacement rank $r$) if there are vectors $s,t$ and matrices $G \in \C^{n\times r}, B \in \C^{r\times n}$ such that
\begin{equation}\label{eq:disp}
 \nabla_{s,t}(C)=GB.
\end{equation}
Notice that if we allow $r=n$, then any matrix is Cauchy-like. In the applications, we are usually interested in cases in which $r \ll n$, since the computational cost of all the involved algorithms depends on $r$.

A Cauchy-like matrix is called a \emph{quasi-Cauchy matrix} if $r=1$, and \emph{Cauchy matrix} if $G^*=B=(1,1,\dotsc,1)$. Usually, it is assumed that the operator $\nabla_{t,s}$ is nonsingular, or equivalently, $t_i \neq s_j$ for all pairs $i,j$. Under this assumption, the elements of $C$ can be written explicitly as
\begin{equation}\label{eq:cauchyformula}
 C_{ij}=\frac{\sum_{l=1}^r G_{il} B_{lj}}{t_i-s_j},
\end{equation}
thus $C$ can be fully recovered from $G$, $B$, $t$ and $s$. Otherwise, the latter formula only holds for the entries $C_{ij}$ such that $t_i \neq s_j$, and $C$ is said to be \emph{partially reconstructible}. The matrices $G$ and $B$ are called the \emph{generators} of $C$, and the elements of $t$ and $s$ are called \emph{nodes}. The vectors $t$ and $s$ are called \emph{node vectors}, or \emph{displacement vectors}.

\paragraph{Trummer-like matrices} In \autoref{s:trummer}, we will deal with the case in which $t=s$ is injective, that is, when the non-reconstructible elements are exactly the ones belonging to the main diagonal. We will use $\disp_s$ as a shorthand for $\disp_{s,s}$. If $\nabla_s(T)=GB$ has rank $r$, a matrix $T$ will be called \emph{Trummer-like} (with displacement rank $r$). Notice that a Trummer-like matrix can be fully recovered from $G$, $B$, $s$, and $d=\diag(T)$. Trummer-like matrices are related to interpolation problems \cite{gera}, and may arise from the transformation of Toeplitz and similar displacement structure \cite{ko}, or directly from the discretization of differential problems \cite{bip}.

\section{Overview of the GKO Schur step}\label{s:gko}
\paragraph{Derivation} The fast LU factorization of a Cauchy-like matrix $C$ is based on the following lemma.
\begin{lemma}\cite{ks}
 Let \[C=\twotwo{C_{1,1}}{C_{1,2\col n}}{C_{2\col n,1}}{C_{2\col n, 2\col n}}\] satisfy the displacement equation \eqref{eq:disp}, and suppose $C_{1,1}$ nonsingular. Then its Schur complement $C^{(2)}=C_{2\col n, 2\col n}-C_{2\col n,1} {C_{1,1}}^{-1} C_{1,2\col n}$ satisfies the displacement equation
\[
 \diag(t_{2\col n}  )C^{(2)}-C^{(2)}\diag(s_{2 \col n} )= G^{(2)} B^{(2)},
\]
with
\begin{equation}\label{eq:genupdate}
 G^{(2)}=G_{2\col n,1 \col r}  -C_{2\col n,1}{C_{1,1}}^{-1}G_{1,1 \col r},\, B^{(2)}=B_{1 \col r, 2 \col n} - B_{1 \col r, 1}  {C_{1,1}}^{-1} C_{1,2 \col n}.
\end{equation}
\end{lemma}
Using this lemma, we can construct the LU factorization of $C$ with $O(n^2)$ floating point operations (ops). The algorithm goes on as follows. Given $G^{(1)}=G$, $B^{(1)}=B$, and the two vectors $s$ and $t$, recover the pivot $C_{1,1}$, the first row $C_{1,2\col n}$ and the first column $C_{2 \col n, 1}$ of $C$ using the formula \eqref{eq:cauchyformula}. This allows to calculate easily the first row of $U$ as $\rowvett{C_{1,1}}{C_{1,2 \col n}}$ and the first column of $L$ as $\rowvett{1}{C_{2 \col n,1}^T {C_{1,1}}^{-1}}^T$. Then use equations \eqref{eq:genupdate} to obtain the generators $G^{(2)}$ and $B^{(2)}$ of the Schur complement $C^{(2)}$ of $C$. Repeat the algorithm setting $G \leftarrow G^{(2)}$, $B \leftarrow B^{(2)}$, $s \leftarrow s_{2 \col n}$ and $t \leftarrow t_{2 \col n}$ to get the second row of $U$ and the second column of $L$, and so on. A simple implementation is outlined in \autoref{a:lu}. Note that for the sake of clarity we used two different variables $L$ and $U$; in fact, it is a widely used technique to have them share the same $n \times n$ array, since only the upper triangular part of the matrix is actually used in $U$, and only the strictly lower triangular part in $L$.
\begin{algorithm}[ht]
\begin{algorithmic}
\REQUIRE $G \in \C^{n \times r},\,B\in\C^{r\times n},\,t,s \in \C^n$ \COMMENT{generators of the matrix}\\
\COMMENT{temporary variables: $L,U \in C^{n \times n}$ (can share the same storage space)}
\STATE $L \leftarrow I_n,\,U \leftarrow O_n$
\FOR{$k=1$ to $n-1$}
\STATE $U_{k,\ell}\leftarrow \frac{G_{k,\col}B_{\col,\ell}}{t_k-s_\ell}$ for all $\ell=k$ to $n$
\STATE $L_{\ell,k}\leftarrow U_{k,k}^{-1}\frac{G_{\ell,\col}B_{\col,k}}{t_\ell-s_k}$ for all $\ell=k+1$ to $n$

\STATE $G_{\ell,\col} \leftarrow G_{\ell,\col}-L_{\ell,k}G_{k,\col}$ for all $\ell=k+1$ to $n$
\STATE $B_{\col,\ell} \leftarrow B_{\col,\ell} - U_{k,k}^{-1} B_{\col,k} U_{k,\ell}$ for all $\ell=k+1$ to $n$

\ENDFOR
\STATE $U_{n,n}\leftarrow \frac{G_{n,\col}B_{\col,n}}{t_n-s_n}$

\RETURN $L,U$
\end{algorithmic}
\caption{LU factorization of Cauchy-like matrices \cite{gko95}}\label{a:lu}
\end{algorithm}
When the LU factorization is only used for the solution of a linear system in the form $Cx=b$, with $b \in \C^{n \times m}$, it is a common technique to avoid constructing explicitly $L$, computing instead $L^{-1}b$ on-the-fly as the successive columns of $L$ are computed. This is also possible with the GKO algorithm, as shown in \autoref{a:implicitl}.
\begin{algorithm}[ht]
\begin{algorithmic}
\REQUIRE $G \in \C^{n \times r},\,B\in\C^{r\times n},\,t,s \in \C^n$ \COMMENT{generators of the matrix}
\REQUIRE $b \in \C^{n \times m}$ \COMMENT{right-hand side}\\
\COMMENT{temporary variables: $l\in\C^n,\,U \in \C^{n\times n}$}
\STATE $U \leftarrow O_n$
\STATE $x \leftarrow b$
\FOR{$k=1$ to $n-1$}
\STATE $l_{\ell}\leftarrow \frac{G_{\ell,\col}B_{\col,k}}{t_\ell-s_k}$ for all $\ell=k$ to $n$ 
\STATE $q\leftarrow\argmax_{\ell=k,k+1,\dots,n} \abs{l_{\ell}}$ \COMMENT{Finds pivot position}
\STATE $p\leftarrow l_q$ \COMMENT{pivot}
\IF{p=0} \PRINT 'error: singular matrix'
\ENDIF
\STATE swap $l_k$ and $l_q$; $x_{k,\col}$ and $x_{q,\col}$; $G_{k,\col}$ and $G_{q,\col}$; $t_{k}$ and $t_{q}$
\STATE $U_{k,k}\leftarrow p$
\STATE $U_{k,\ell}\leftarrow \frac{G_{k,\col}B_{\col,\ell}}{t_k-s_\ell}$ for all $\ell=k+1$ to $n$
\STATE $x_{\ell,\col}\leftarrow x_{\ell,\col}- l_{\ell} (p^{-1}x_{k,\col})$ for all $\ell=k+1$ to $n$
\STATE $G_{\ell,\col}\leftarrow G_{\ell,\col}- l_{\ell} (p^{-1}G_{k,\col})$ for all $\ell=k+1$ to $n$
\STATE $B_{\col,\ell}\leftarrow B_{\col,\ell}-p^{-1} B_{\col,k}U_{k,\ell}$ for all $\ell=k+1$ to $n$
\ENDFOR
\STATE $U_{n,n}\leftarrow \frac{G_{n,\col}B_{\col,n}}{t_n-s_n}$
\STATE $x_{n,\col}\leftarrow x_{n,\col}/U_{n,n}$ \COMMENT{start of the back-substitution step}
\FOR{$k=n-1$ down to $1$}
\STATE $x_{k,\col} \leftarrow x_{k,\col}-U_{k,\ell}x_{\ell,\col}$ for all $\ell=k+1$ to $n$
\STATE $x_{k,\col}\leftarrow x_{k,\col}/U_{k,k}$
\ENDFOR
\RETURN $x$
\end{algorithmic}
\caption{Solving a system $Cx=b$ with implicit $L$ factor and pivoting \cite{gko95}}\label{a:implicitl}
\end{algorithm}

\paragraph{Comments} Notice that \autoref{a:implicitl} includes partial pivoting. Its total cost is $(4r+2m+1)n^2+o(n^2)$ ops, when applied to a matrix $C$ with displacement rank $r$ and an $n \times m$ right-hand side. The algorithm works whenever $C$ is a completely reconstructible Cauchy matrix; if it is not the case, when the number of non-reconstructible entries is small, the algorithm can be modified to store and update them separately, see e.g. Kailath and Olshevsky \cite{ko} or \autoref{s:trummer}.

However, there is an important drawback in \autoref{a:implicitl}: while the size of the input and output data is $O(n)$ (for small values of $m$ and $r$), $O(n^2)$ memory locations of temporary storage are needed along the algorithm to store $U$. Therefore, for large values of $n$ the algorithm cannot be effectively implemented on a computer because it does not fit in the RAM. 

Moreover, another important issue is caching. Roughly speaking, a personal computer has about 512 kb--8 Mb of cache memory, where the most recently accessed locations of RAM are copied. Accessing a non-cached memory location is an order of magnitude slower than a cached one. The real behavior of a modern processor is more complicated than this simple model, due to the presence of several different levels of cache, each with its own performance, and instruction pipelines \cite{HennessyPatterson}. Nevertheless, this should highlight that when the used data do not fit anymore into the cache, saving on memory could yield a greater speedup than saving on floating point operations.

\section{Low-storage version of GKO: the extended matrix approach}\label{s:ar}
\paragraph{Derivation}
The following algorithm to solve the high storage issue in GKO was proposed by Rodriguez \cite{rod} in 2006, while dealing with least squares Cauchy-like problems. More recently, a deeper analysis and a ready-to-use Matlab implementation were provided by Aric\`o and Rodriguez \cite{ar}.

The approach is based on an idea that first appeared in Kailath and Chun \cite{kc}. Let us suppose that $C$ is a completely reconstructible Cauchy-like matrix and that $s$ is injective. The solution of the linear system $Cx=b$ can be expressed as the Schur complement of $C$ in the rectangular matrix
\[
 \widetilde C = \twotwo{C}{b}{-I}{0}.
\]
Thus, we can compute $x$ by doing $n$ steps of Gaussian elimination on $\widetilde C$. Moreover, the first block column of $\widetilde C$ is a partially reconstructible Cauchy-like matrix with respect to $\widetilde s=\rowvett{t^T}{s^T}^T$ and $t$; therefore, while performing the Gaussian elimination algorithm, the entries of this block can be stored and updated in terms of the generators, as in \autoref{a:lu}. Unlike the previous algorithms, we may discard the rows of $U$ and columns of $L$ as soon as they are computed, keeping only the generators. Instead, the entries in the second block column are computed with  customary Gaussian elimination and stored along all the algorithm.

The following observations, which will be needed later, should make clearer what is going on with this approach.
\begin{lemma}\label{th:em-structure}
Suppose for simplicity that no pivoting is performed; let $L$ and $U$ be the LU factors of $C$, $x$ be the solution to the linear system $Cx=b$, $y$ be the solution to $Ly=b$, and $W=U^{-1}$. Let $k$ denote the step of Gaussian elimination being performed, with e.g. $k=1$ being the step that zeroes out all the elements of the first column but the first.
During the algorithm,
 \begin{enumerate}
 \item The $(i,j)$ entry of the $(1,1)$ block is updated at all steps $k$ with $k<\min(i,j+1)$. After its last update, it contains $U_{i,j}$.
 \item The $(i,j)$ entry of the $(1,2)$ block is updated at all steps $k$ with $k<i$. After its last update, it contains $y_{i,j}$.
 \item The $(i,j)$ entry of the $(2,2)$ block is updated at all steps $k$ with $k \geq i$. In particular, the last step ($k=n$) updates all entries, and after that the $(2,2)$ block contains $x_{i,j}$.
 \item The $(i,j)$ entry of the $(2,1)$ block is updated at all steps $k$ with $i \leq k\leq j$. After its last update, it contains $0$. Immediately before that, i.e., just after step $j-1$, it contains $-W_{i,j}U_{j,j}$.
 \end{enumerate}
\end{lemma}
\begin{proof}
 From the structure of Gaussian elimination, it can easily be verified that the entries are only updated during the abovementioned steps. In particular, for the condition on updates to the $(2,1)$ block, it is essential that the initial $(2,1)$ block initially contains a diagonal matrix. Regarding which values appear finally in each position,
\begin{enumerate}
 \item is obvious: in fact, if we ignore all the other blocks, we are doing Gaussian elimination on $C$.
 \item is easily proved: since the row operations we perform transform $C=LU$ to $U$, they must be equivalent to left multiplication by $L^{-1}$.
 \item is a consequence of the well-known fact that after $n$ steps of Gaussian elimination we get the Schur complement of the initial matrix in the trailing diagonal block.
 \item is less obvious. Let us call $Z_{i,k}$ the value of the $(i,k)$ entry of the $(2,1)$ block right after step $k-1$, and consider how the entries of the $(2,2)$ block are updated along the algorithm. They are initially zero, and at the $k$th step the one in place $(i,j)$ is incremented by $-(Z_{i,k}/U_{k,k})y_{k,j}$, so its final value is
\[
x_{i,j}=-\sum_k (Z_{i,k}/U_{k,k})y_{k,j}.
\]
Since for each choice of $b$ (and thus of $y=L^{-1}b$) $Z_{i,k}$ and $W_{i,k}$ are unchanged, as they only depend on $C$, and it holds that
\[
 x_{i,j}=(U^{-1}y)_{i,j}= \sum_k W_{i,k}y_{k,j},
\]
the only possibility is that $W_{i,k}=-Z_{i,k}/U_{k,k}$ for each $i$, $k$.
\end{enumerate}
\end{proof}
We report here the resulting \autoref{a:ar}.
\begin{algorithm}[ht]
\begin{algorithmic}
\REQUIRE $G \in \C^{n \times r},\,B\in\C^{r\times n},\,t,s \in \C^n$ \COMMENT{generators of the matrix}
\REQUIRE $b \in \C^{n \times m}$ \COMMENT{right-hand side}\\
\COMMENT{temporary variables: $l,u\in\C^n$}
\STATE $x \leftarrow b$
\FOR{$k=1$ to $n-1$}
\STATE $l_{\ell}\leftarrow \frac{G_{\ell,\col}B_{\col,k}}{s_\ell-s_k}$ for all $\ell=1$ to $k-1$
\STATE $l_{\ell}\leftarrow \frac{G_{\ell,\col}B_{\col,k}}{t_\ell-s_k}$ for all $\ell=k$ to $n$
\STATE $q\leftarrow\argmax_{\ell=k,k+1,\dots,n} \abs{l_{\ell}}$ \COMMENT{Finds pivot position}
\STATE $p\leftarrow l_q$ \COMMENT{pivot}
\IF{p=0} \PRINT 'error: singular matrix'
\ENDIF
\STATE swap $l_k$ and $l_q$; $x_{k,\col}$ and $x_{q,\col}$; $G_{k,\col}$ and $G_{q,\col}$; $t_{k}$ and $t_{q}$
\STATE $u_{\ell}\leftarrow \frac{G_{k,\col}B_{\col,\ell}}{t_k-s_\ell}$ for all $\ell=k+1$ to $n$
\STATE $x_{k,\col}\leftarrow p^{-1}x_{k,\col}$
\STATE $x_{\ell,\col}\leftarrow x_{\ell,\col}-l_{\ell}x_{k,\col}$ for all $\ell \neq k$
\STATE $G_{k,\col}\leftarrow p^{-1}G_{k,\col}$
\STATE $G_{\ell,\col}\leftarrow G_{\ell,\col}-l_{\ell}G_{k,\col}$ for all $\ell \neq k$
\STATE $B_{\col,\ell}\leftarrow B_{\col,\ell}-p^{-1} B_{\col,k}u_{\ell}$ for all $\ell=k+1$ to $n$
\ENDFOR
\STATE $l_{\ell}\leftarrow \frac{G_{\ell,\col}B_{\col,k}}{s_\ell-s_k}$ for all $\ell=1$ to $n-1$
\STATE $p\leftarrow \frac{G_{n,\col}B_{\col,n}}{t_n-s_n}$
\STATE $x_{n,\col}\leftarrow p^{-1}x_{n,\col}$
\STATE $x_{\ell,\col}\leftarrow x_{\ell,\col}-l_{\ell}x_{n,\col}$ for all $\ell \neq n$
\RETURN $x$
\end{algorithmic}

\caption{Solving a system $Cx=b$ with the extended matrix algorithm \cite{ar}}\label{a:ar}
\end{algorithm}
\paragraph{Comments}
It is worth mentioning that several nice properties notably simplify the implementation.
\begin{itemize}
 \item The partial reconstructibility of $\widetilde C$ is not an issue. If the original matrix $C$ is fully reconstructible and $s$ is injective, then the non-reconstructible entries of $\widetilde C$ are the ones in the form $C(n+k,k)$ for $k=1,\dots,n$, that is, the ones in which the $-1$ entries of the $-I$ block initially lie. It is readily shown that whenever the computation of such entries is required, their value is the initial one of $-1$.
 \item At each step of the algorithm, the storage of only $n$ rows of $G$ and of the right block column $x$ is required: at step $k$, we only need the rows with indices from $k$ to $n+k-1$ (as the ones below are still untouched by the algorithm, and the ones above are not needed anymore). It is therefore possible to reuse the temporary variables to store the rows modulo $n$, thus halving the storage space needed for some of the matrices.
 \item Pivoting can be easily included without destroying the block structure by acting only on the rows belonging to the first block row of $\widetilde C$.
\end{itemize}

\autoref{a:ar} uses $(6r+2m+\frac 32)n^2+o(n^2)$ floating point operations, and it can be implemented so that the input variables $G$, $B$, $t$, $b$ are overwritten during the algorithm, with $x$ overwriting $b$, so that it only requires $2n$ memory locations of extra storage (to keep $l$ and $u$).

As we stated above, for the algorithm to work we need the additional assumption that $s$ is injective, i.e., $s_i \neq s_j$ for all $j$. This is not restrictive when working with Cauchy-like matrices derived from Toeplitz matrices or from other displacement structured matrices; in fact, in this case the entries $s_i$ are the $n$ complex $n$th roots of a fixed complex number, thus not only are they different, but their differences $s_i-s_j$ can be easily bounded from below, which is important to improve the stability of the algorithm. This is a common assumption when dealing with Cauchy matrices, since a Cauchy (or quasi-Cauchy) matrix is nonsingular if and only if $x$ and $y$ are injective. For Cauchy-like matrices this does not hold, but the injectivity of the two vectors is still related to the singularity of the matrix: for instance, we have the following result.
\begin{lemma}
Let $s$ have $r+1$ repeated elements, that is, $s_{i_1}=s_{i_2}=\dots=s_{i_{r+1}}=s$. Then the Cauchy-like matrix \eqref{eq:cauchyformula} is singular.
\end{lemma}
\begin{proof}
Consider the submatrix $C'$ formed by the $r+1$ columns of $C$ with indices $i_1,\dotsc,i_{r+1}$. It is the product of the two matrices $G' \in \C^{n \times r}$ and $B' \in \C^{r \times r+1}$, with
\[
 (G')_{ij}=\frac{G_{ij}}{t_i-s}, \, (B')_{ij}=B_{is_j}.
\]
Therefore $C'$ (and thus $C$) cannot have full rank.
\end{proof}

\section{Low-storage version of GKO: the downdating approach}\label{s:dd}
\paragraph{Derivation} In this section, we shall describe a different algorithm to solve a Cauchy-like system using only $O(n)$ locations of memory. Our plan is to perform the first \textbf{for} loop in \autoref{a:implicitl} unchanged, thus getting $y=L^{-1}b$, but discarding the computed entries of $U$ which would take $O(n^2)$ memory locations, and then to recover them via additional computations on the generators.

For the upper triangular system $Ux=y$ to be solved incrementally by back-substitution, we need the entry of the matrix $U$ to be available one row at a time, starting from the last one, and \emph{after} the temporary value $y=L^{-1}b$ has been computed, that is, after the whole $LU$ factorization has been performed.

Let $G^{(k)}$ ($B^{(k)}$) denote the contents of the variable $G$ (resp. $B$) after step $k$. The key idea is trying to undo the transformations performed on $B$ step by step, trying to recover $B^{(k)}$ from $B^{(k+1)}$. Because of the way in which the generators are updated in Algorithms \ref{a:lu} and \ref{a:implicitl}, the first row of $G^{(k)}$ and the first column of $B^{(k)}$ are kept in memory untouched by iterations $k+1,\dotsc,n$ of the GKO algorithm. Thus we can use them in trying to undo the $k$th step of Gaussian elimination.

Let us suppose we know $B^{(k+1)}$, i.e., the contents of the second generator $B$ after the $(k+1)$st step of Gaussian elimination, and the values of $G^{(k)}_{k,\col}$ and $B^{(k)}_{\col,k}$, which are written in $G$ and $B$ by the $k$th step of Gaussian elimination and afterwards unmodified (since the subsequent steps of \autoref{a:implicitl} do not use those memory locations anymore).

We start from the second equation of \eqref{eq:genupdate} and \eqref{eq:cauchyformula} for the $k$th row of $U$, written using the colon notation for indices.
\[
\begin{aligned}
 B^{(k+1)}_{\col,\ell} &= B^{(k)}_{\col,\ell} - B^{(k)}_{\col,k}{U_{k,k}}^{-1}U_{k,\ell}, & \ell &> k,\\
 U_{k,\ell}&=\frac{G^{(k)}_{k,\col} B^{(k)}_{\col,\ell}}{t_k-s_\ell}, & \ell &\geq k.
\end{aligned}
\]
Substituting $B^{(k)}_{\col,\ell}$ from the first into the second, and using the $k=\ell$ case of the latter to deal with $U_{k,k}$, we get
\[
 U_{k,\ell}=\frac{G^{(k)}_{k,\col} B^{(k+1)}_{\col,\ell}}{t_k-s_\ell}+\frac{G^{(k)}_{k,\col} B^{(k)}_{\col,k} {U_{k,k}}^{-1}}{t_k-s_\ell} U_{k,\ell} = \frac{G^{(k)}_{k,\col} B^{(k+1)}_{\col,\ell}}{t_k-s_\ell} + \frac{t_k-s_k}{t_k-s_\ell}U_{k,\ell}
\]
and thus
\begin{align}
 U_{k,\ell}&=\frac{G^{(k)}_{k,\col} B^{(k+1)}_{\col,\ell}}{s_k-s_\ell}, & \ell &\geq k, \label{uformula}\\
  B^{(k)}_{\col,\ell} &= B^{(k+1)}_{\col,\ell} + B^{(k)}_{\col,k}{U_{k,k}}^{-1}U_{k,\ell}, & \ell &> k.
\end{align}
The above equations allow one to recover the value of $B^{(k)}_{\col,\ell}$ for all $\ell >k$ using only $B^{(k+1)}_{\col,\ell}$, $G^{(k)}_{k,\col}$ and $B^{(k)}_{\col,k}$ as requested.

By applying the method just described repeatedly for $k=n-1,n-2,\dots,1$, we are able to recover one at a time the contents of $B^{(n-1)},B^{(n-2)},\dots, B^{(1)}$, which were computed (and then discarded) in the first phase of the algorithm. I.e., at each step we ``downdate'' $B$ to its previous value, reversing the GKO step. In the meantime, we get at each step $k=n-1,n-2,\dots,1$ the $k$th row of $U$. In this way, the entries of $U$ are computed in a suitable way to solve the system $Ux=y$ incrementally by back-substitution.

 We report here the resulting \autoref{a:dd}.
\begin{algorithm}[ht]
\begin{algorithmic}
\REQUIRE $G \in \C^{n \times r},\,B\in\C^{r\times n},\,t,s \in \C^n$ \COMMENT{generators of the matrix}
\REQUIRE $b \in \C^{n \times m}$ \COMMENT{right-hand side}\\
\COMMENT{temporary variables: $l,u\in\C^n$}
\STATE $x \leftarrow b$
\FOR{$k=1$ to $n-1$}
\STATE $l_{\ell}\leftarrow \frac{G_{\ell,\col}B_{\col,k}}{t_\ell-s_k}$ for all $\ell=k$ to $n$ 
\STATE $q\leftarrow\argmax_{\ell=k,k+1,\dots,n} \abs{l_{\ell}}$ \COMMENT{Finds pivot position}
\STATE $p\leftarrow l_q$ \COMMENT{pivot}
\IF{p=0} \PRINT 'error: singular matrix'
\ENDIF
\STATE swap $l_k$ and $l_q$; $x_{k,\col}$ and $x_{q,\col}$; $G_{k,\col}$ and $G_{q,\col}$; $t_{k}$ and $t_{q}$
\STATE $u_k\leftarrow p$
\STATE $u_{\ell}\leftarrow \frac{G_{k,\col}B_{\col,\ell}}{t_k-s_\ell}$ for all $\ell=k+1$ to $n$
\STATE $x_{\ell,\col}\leftarrow x_{\ell,\col}-p^{-1} l_{\ell} x_{k,\col}$ for all $\ell=k+1$ to $n$
\STATE $G_{\ell,\col}\leftarrow G_{\ell,\col}-p^{-1} l_{\ell} G_{k,\col}$ for all $\ell=k+1$ to $n$
\STATE $B_{\col,\ell}\leftarrow B_{\col,\ell}-p^{-1} B_{\col,k}u_{\ell}$ for all $\ell=k+1$ to $n$
\ENDFOR
\STATE $u_{n}\leftarrow \frac{G_{n,\col}B_{\col,n}}{t_n-s_n}$
\STATE $x_{n,\col}\leftarrow x_{n,\col}/u_{n}$ \COMMENT{start of the back-substitution step}
\FOR{$k=n-1$ down to $1$}
\STATE $u_\ell\leftarrow \frac{G_{k,\col}B_{\col,\ell}}{s_k-s_\ell}$ for all $\ell=k+1$ to $n$
\STATE $B_{\col,\ell} \leftarrow B_{\col,\ell}+u_k^{-1}B_{\col,k}u_\ell$ for all $\ell=k+1$ to $n$
\STATE $x_{k,\col} \leftarrow x_{k,\col}-u_{\ell}x_{\ell,\col}$ for $\ell=k+1$ to $n$
\STATE $x_{k,\col}\leftarrow x_{k,\col}/u_{k}$
\ENDFOR
\RETURN $x$ 
\end{algorithmic}
\caption{Solving a system $Cx=b$ with the downdating algorithm}\label{a:dd}
\end{algorithm}
\paragraph{Comments} Notice that pivoting only affects the first phase of the algorithm, since the whole reconstruction stage can be performed on the pivoted version of $C$ without additional row exchanges.

This algorithm has the same computational cost, $(6r+2m+\frac 32)n^2+o(n^2)$, and needs the same number of memory locations, $2n$, as the extended matrix approach. Moreover, they both need the additional property that $s$ be injective, as an $s_k-s_\ell$ denominator appears in \eqref{uformula}. These facts may lead one to suspect that they are indeed the same algorithm. However, it is to be noted the two algorithms notably differ in the way in which the system $Ux=y$ is solved: in the extended matrix approach we solve this system by accumulating the explicit multiplication $U^{-1}y$, while in the downdating approach we solve it by back-substitution. 

Several small favorable details suggest adopting the latter algorithm:
\begin{itemize}
 \item With the extended matrix approach, we do not get any entry of $x$ before the last step. On the other hand, with the downdating approach, as soon as the first \textbf{for} cycle is completed, we get $x_{n}$, and then after one step of the downdating part we get $x_{n-1}$, and so on, getting one new component of the solution at each step. This is useful because in the typical use of this algorithm on Toeplitz matrices, $x$ is the Fourier transform of a ``meaningful'' vector, such as one representing a signal, or an image, or the solution to an equation. Using the correct ordering, the last entries of a Fourier transform can be used to reconstruct a lower-sampled preview of the original data, with no additional computational overhead, see e.g. Walker\cite{walker}. Thus with this approach we can provide an approximate solution after only the first part of the algorithm is completed.
 \item In the extended matrix version, each step of the algorithm updates $O(nr)$ memory locations. Instead, in the downdating version, for each $k$, the $(n-k)$th and $(n+k)$th step work on $O(kr)$ memory locations. Therefore, the ``innermost'' iterations take only a small amount of memory and thus fit better into the processor cache. This is a desirable behavior similar to the one of \emph{cache-oblivious algorithms} \cite{cacheob}.
 \item In exact arithmetic, at the end of the algorithm the second generator $B$ of the matrix $C$ is reconstructed as it was before the algorithm. In floating point arithmetic, this can be used as an \emph{a posteriori} accuracy test: if one or more entries of the final values of $B$ are not close to their initial value, then there was a noticeable algorithmic error.
\end{itemize}

\section{Computations with Trummer-like matrices}\label{s:trummer}
A special class of Cauchy-like matrices which may arise in application \cite{gera,ko,bip,bmp_cr} is that of Trummer-like matrices, i.e., those for which the two node vectors coincide ($t=s$) and are injective. The partial reconstructibility of these matrices requires special care to be taken in the implementation of the solution algorithms. The $O(n)$-storage algorithms we have presented can be adapted to deal with this case. Moreover, it is a natural request to ask for an algorithm that computes the generators of $T^{-1}$ given those of $T$. Such an algorithm involves three different parts: the solution of a linear system with matrix $T$ and multiple right-hand side; the solution of a similar system with matrix $T^*$, and the computation of $\diag(T^{-1})$. We will show that an adaptation of the algorithm presented in \autoref{s:ar} can perform all three at the same time, fully exploiting the fact that these three computations share a large part of the operations involved.

\paragraph{Theoretical results} The following results, which are readily proved by expanding the definition of $\disp_{s}$ on both sides, are simply the adaptation of classical results on displacement ranks (see e.g. Heinig and Rost \cite{hr}) to the Trummer-like case. Notice the formal similarity with the derivative operator.
\begin{theorem}\label{derivate}
 Let $A,B \in \C^{n \times n}$, and let $r(X)=\operatorname{rk} \disp_s(A)$.
\begin{enumerate}
 \item $\disp_s(A+B)=\disp_s(A)+\disp_s(B)$, so $r(A+B) \leq r(A)+r(B)$.
 \item $\disp_s(AB)=\disp_s(A)B+A\disp_s(B)$, so $r(AB) \leq r(A)+r(B)$.
 \item $\disp_s(A^{-1})=-A^{-1}\disp_s(A)A^{-1}$, so $r(A^{-1})=r(A)$.
\end{enumerate}
\end{theorem}
As we saw in \autoref{s:basic}, a Trummer-like matrix can be completely reconstructed by knowing only the node vector $s$, the generators $G$ and $B$, and its diagonal $d=\diag(T)$. In this section, we are interested in implementing fast---i.e., using $O(n^2)$ ops---and space-efficient---i.e., using $O(n)$ memory locations---matrix-vector and matrix-matrix operations involving Trummer-like matrices stored in this form.

\paragraph{Matrix-vector product} For the matrix-vector product, all we have to do is reconstructing one row at a time of the matrix $T$ and then computing the customary matrix-vector product via the usual formula $(Tv)_i=\sum_j T_{ij}v_j$. Approximate algorithms for the computation of the Trummer-like matrix-vector product with $O(n \log^2 n)$ ops also exist, see e.g. Bini and Pan \cite{binipan}.

\paragraph{Matrix-matrix operations} The matrix product between two Trummer-like matrices $T$ and $S$ is easy to implement: let $G_T$ and $B_T$ (resp. $G_S$ and $B_S$) be the generators of $T$ (resp. $S$); then, by \autoref{derivate}, the generators of $TS$ are
\[
\rowvett{TG_S}{G_T},\quad\vett{B_S}{B_TS},
\]
while $\diag(TS)$ can be computed in $O(n^2)$ by recovering at each step one row of $T$ and one column of $S$ and computing their dot product. Sums are similar: the generators of $S+T$ are
\[
\rowvett{G_S}{G_T},\quad\vett{B_S}{B_T},
\]
and its diagonal is $d_S+d_T$.

\paragraph{Linear systems} Linear system solving is less obvious. Kailath and Olshevsky \cite{ko} suggested the following algorithm: the GKO Gaussian elimination is performed, but at the same time the computed row $U_{k,k\col n}$ and column $L_{k\col n, k}$ are used to update the diagonal $d$ to the diagonal of the Schur complement, according to the customary Gaussian elimination formula
\begin{equation}\label{gaussdiag}
T_{i,i}^{(k+1)} = T^{(k)}_{i,i}-L_{i,k}(T^{(k)}_{k,k})^{-1}U_{k,i}.
\end{equation}
It is easy to see that this strategy can be adapted to both the extended matrix and the downdating version of the algorithm, thus allowing one to implement GKO with $O(n)$ storage also for this class of matrices.

However, a more delicate issue is pivoting. Kailath and Olshevsky do not deal with the general case, since they work with symmetric matrices and with a symmetric kind of pivoting that preserves the diagonal or off-diagonal position of the entries. Let us consider the pivoting operation before the $k$th step of Gaussian elimination, which consists in choosing an appropriate row $q$ and exchanging the $k$th and $q$th rows. The main issue here is that the two non-reconstructible entries that were in position $T_{kk}$ and $T_{qq}$, now are in positions $T_{qk}$ and $T_{kq}$. This requires special handling in the construction of the $k$th row in the Gaussian elimination step, but luckily it does not affect the successive steps of the algorithm, since the $k$th column and row are not used from step $k+1$ onwards. On the other hand, the entry $T_{qq}$, which used to be non-reconstructible before pivoting, is now reconstructible. We may simply ignore this fact, store it in $d$ and update it with the formula \eqref{gaussdiag} as if it were not reconstructible. The algorithm is given here as \autoref{a:trsv}.
\begin{algorithm}[ht]
\begin{algorithmic}
\REQUIRE $G \in \C^{n \times r},\,B\in\C^{r\times n},\,s \in \C^n$ \COMMENT{generators of the matrix $T$}
\REQUIRE $d \in \C^n$ \COMMENT{diagonal of $T$}
\REQUIRE $b \in \C^{n \times m}$ \COMMENT{right-hand side}\\
\COMMENT{temporary variables: $l,u\in\C^n$}\\
\COMMENT{temporary variable: $\sigma \in \mathbb N^n$ vector of integer indices used to keep track of the permutation performed during the pivoting}
\STATE $\sigma(i)\leftarrow i$ for all $i=1$ to $n$
\COMMENT{initializes $\sigma$ as the identity permutation}
\FOR{$k=1$ to $n-1$}
\STATE $l_k\leftarrow d_k$
\STATE $l_{\ell}\leftarrow \frac{G_{\ell,\col}B_{\col,k}}{s_{\sigma(\ell)}-s_k}$ for all $\ell=k+1$ to $n$ 
\STATE $q\leftarrow\argmax_{\ell=k,k+1,\dots,n} \abs{l_{\ell}}$ \COMMENT{Finds pivot position}
\STATE $p\leftarrow l_q$ \COMMENT{pivot}
\IF{p=0} \PRINT 'error: singular matrix'
\ENDIF
\STATE swap $l_k$ and $l_q$; $x_{k,\col}$ and $x_{q,\col}$; $G_{k,\col}$ and $G_{q,\col}$; $\sigma(k)$ and $\sigma(q)$
\STATE $u_k\leftarrow p$
\STATE $u_{\ell}\leftarrow \frac{G_{k,\col}B_{\col,\ell}}{s_{\sigma(k)}-s_\ell}$ for all $\ell=k+1$ to $n$, $\ell\neq q$
\STATE $u_q\leftarrow d_q$ \COMMENT{non-reconstructible entry that moved off-diagonal after pivoting}
\STATE $x_{\ell,\col}\leftarrow x_{\ell,\col}-p^{-1} l_{\ell} x_{k,\col}$ for all $\ell=k+1$ to $n$
\STATE $G_{\ell,\col}\leftarrow G_{\ell,\col}-p^{-1} l_{\ell} G_{k,\col}$ for all $\ell=k+1$ to $n$
\STATE $B_{\col,\ell}\leftarrow B_{\col,\ell}-p^{-1} B_{\col,k}u_{\ell}$ for all $\ell=k+1$ to $n$
\STATE $d_\ell\leftarrow d_\ell-p^{-1}l_\ell u_\ell$ for all $\ell=k+1$ to $n$\COMMENT{Gaussian elimination on the diagonal}
\IF[$d_q$ may be reconstructible after the pivoting --- but we store it explicitly anyway]{$q\neq k$}
\STATE $d_q\leftarrow\frac{G_{q,\col}B_{\col,q}}{s_{\sigma(q)}-s_q}$
\ENDIF
\ENDFOR
\STATE $u_{n}\leftarrow \frac{G_{n,\col}B_{\col,n}}{s_{\sigma(n)}-s_n}$
\STATE $x_{n,\col}\leftarrow x_{n,\col}/u_{n}$ \COMMENT{start of the back-substitution step}
\FOR{$k=n-1$ down to $1$}
\STATE $u_\ell\leftarrow \frac{G_{k,\col}B_{\col,\ell}}{s_k-s_\ell}$ for all $\ell=k+1$ to $n$
\STATE $B_{\col,\ell} \leftarrow B_{\col,\ell}+u_k^{-1}B_{\col,k}u_\ell$ for all $\ell=k+1$ to $n$
\STATE $x_{k,\col} \leftarrow x_{k,\col}-u_{\ell}x_{\ell,\col}$ for $\ell=k+1$ to $n$
\STATE $x_{k,\col}\leftarrow x_{k,\col}/u_{k}$
\ENDFOR
\RETURN $x$ 
\end{algorithmic}
\caption{Solving a system $Tx=b$ with the downdating algorithm}\label{a:trsv}
\end{algorithm}

A similar modification is also possible with the extended matrix version of the algorithm --- we shall see it in more detail in the next paragraph.

\paragraph{Matrix inversion} Matrix inversion poses an interesting problem too. The generators of $T^{-1}$ can be easily computed as $T^{-1}G$ and $-BT^{-1}$ by resorting to Lemma~\ref{a:trsv} applied twice on $T$ and $T^*$. However, whether we try to compute the representation of $T^{-1}$ or directly that of $T^{-1}S$ for another Trummer-like matrix $S$, we are faced with the problem of computing $\diag(T^{-1})$ given a representation of $T$. There appears to be no simple direct algorithm to extract it in time $O(n^2)$ from the LU factors of $T$.

A possible solution could be based on the decomposition $T^{-1}=\diag(f)+F$, with $f$ a vector and $F$ a matrix with $\diag(F)=[0,0,\dots,0]^T$. In fact, notice that $F$ depends only on the generators of the inverse; therefore, after computing them, one could choose any vector $v$ for which $T^{-1}v$ has already been computed (e.g., $G_{:,1}$) and solve for the entries of $f$ in the equation $T^{-1}v=\diag(f)v+Fv$. This solution was attempted in \cite{bmp_cr}, but was found to have unsatisfying numerical properties.

We shall present here a different solution based on the observations of \autoref{th:em-structure}, that will allow us to compute the diagonal together with the inversion algorithm. Let us ignore pivoting in this first stage of the discussion. Notice that the last part of \autoref{th:em-structure} shows us a way to compute $(U^{-1})_{1 \col k,k}$ at the $k$th step of the extended matrix algorithm. Our plan is to find a similar way to get $(L^{-1})_{k,1\col k}$ at the same step, so that we can compute the sums 
\begin{equation}\label{sumdiag}
(T^{-1})_{i,i}=\sum_k (U^{-1})_{i,k} (L^{-1})_{k,i}, \quad i=1,\dots,n,
\end{equation}
one summand at each step, accumulating the result in a temporary vector.

The following result holds.
\begin{lemma}
 Let $T$ be Trummer-like with generators $G$ and $B$, nodes $s$ and diagonal $d$, $T=LU$ be its LU factorization, and $D=\diag(p)$, where $p_i=U_{i,i}$ are the pivots.
\begin{enumerate}
 \item The LU factorization of $T^*$, the transpose conjugate of $T$, is $(U^*D^{-1})(DL^*)$.
 \item The matrix $T^*$ is Trummer-like with nodes $s$, diagonal $d$ and generators $B^*$ and $G^*$.
 \item Let $G^{(k)}$ and $B^{(k)}$ be the content of the variables $G$ and $B$ after the $k$th step of the GKO algorithm on $T$, and $\bar G^{(k)}$ and $\bar B^{(k)}$ be the content of the same variables after the same step of the GKO algorithm on $T^*$. Then, $\bar G^{(k)}=(B^{(k)})^*$ and $\bar B^{(k)}=(G^{(k)})^*$
\end{enumerate}
\end{lemma}
\begin{proof}
 The matrices $U^*D^{-1}$ and $DL^*$ are respectively unit lower triangular and upper triangular. Thus the first part holds by the uniqueness of the LU factorization. The second part is clear, and the last one follows by writing down the formula \eqref{eq:genupdate} for $T$ and $T^*$.
\end{proof}
Therefore, there is much in common between the GKO algorithm on $T$ and $T^*$, and the two can be carried on simultaneously saving a great part of the computations involved. Moreover, in the same way as we obtain $(U^{-1})_{1 \col k,k}$, we may also get at the $k$th step its equivalent for $T^*$, i.e., $((DL^*)^{-1})_{1 \col ,k}=p_{k}(L^{-1})_{k,1 \col k}$. Since $p_k$, the $k$th pivot, is also known, this allows to recover $(L^{-1})_{k,1\col k}$.

Thus we have shown a way to recover both $(U^{-1})_{1 \col k,k}$ and $(L^{-1})_{k,1 \col k}$ at the $k$th step of the extended matrix algorithm, and this allows to compute the $k$th summand of \eqref{sumdiag} for each $i$.

\paragraph{Pivoting} How does pivoting affect this scheme for the computation of $\diag(T^{-1})$? If $T=PLU$, formula \eqref{sumdiag} becomes
\begin{equation}\label{permvers}
 (T^{-1})_{i,i}=\sum_k (U^{-1})_{i,k} (L^{-1}P^{-1})_{k,i}, \quad i=1,\dots,n.
\end{equation}
The permutation matrix $P$, of which we already have to keep track during the algorithm, acts on $L^{-1}$ by scrambling the column indices $i$, so this does not affect our ability to reconstruct the diagonal, as we still have all the entries needed to compute the $k$th summand at each step $k$. We only need to take care of the order in which the elements of $(U^{-1})_{1 \col k,k}$ and $(L^{-1})_{k,1 \col k}$ are paired in \eqref{permvers}.

The complete algorithm, which includes pivoting, is reported here as \autoref{a:inv}.
\begin{algorithm}[ht]
\begin{algorithmic}
\REQUIRE $G \in \C^{n \times r},\,B\in\C^{r\times n},\,s \in \C^n$ \COMMENT{generators of the matrix $T$}
\REQUIRE $d \in \C^n$ \COMMENT{diagonal of $T$}
\REQUIRE $b \in \C^{n \times m_1},c \in \C^{m_2 \times n}$ \COMMENT{for the (optional) solution of $Tx=b$ and $yT=c$}\\
\COMMENT{temporary variables: $l,u\in\C^n$,$\sigma \in \mathbb N^n$}\\
\STATE $x \leftarrow b;y \leftarrow c$
\STATE $\sigma(i)\leftarrow i$ for all $i=1$ to $n$
\COMMENT{initializes $\sigma$ as the identity permutation}
\FOR{$k=1$ to $n-1$}
\STATE $l_{\ell}\leftarrow \frac{G_{\ell,\col}B_{\col,k}}{s_\ell-s_k}$ for all $\ell=1$ to $k-1$ 
\STATE $l_k\leftarrow d_k$
\STATE $l_{\ell}\leftarrow \frac{G_{\ell,\col}B_{\col,k}}{s_{\sigma(\ell)}-s_k}$ for all $\ell=k+1$ to $n$ 
\STATE $q\leftarrow\argmax_{\ell=k,k+1,\dots,n} \abs{l_{\ell}}$ \COMMENT{Finds pivot position}
\STATE $p\leftarrow l_q$ \COMMENT{pivot}
\IF{p=0} \PRINT 'error: singular matrix'
\ENDIF
\STATE swap $l_k$ and $l_q$; $x_{k,\col}$ and $x_{q,\col}$; $G_{k,\col}$ and $G_{q,\col}$; $\sigma(k)$ and $\sigma(q)$
\STATE $u_\ell \leftarrow \frac{G_{k,\col}B_{\col,\ell}}{s_{\sigma(k)}-s_{\sigma(\ell)}}$ for all $\ell=k+1$ to $n$, $\ell\neq q$
\COMMENT{``extended matrix'' computations for $T^*$}
\STATE $u_{\ell}\leftarrow \frac{G_{k,\col}B_{\col,\ell}}{s_{\sigma(k)}-s_\ell}$ for all $\ell=k+1$ to $n$, $\ell\neq q$
\STATE $u_q\leftarrow d_q$ \COMMENT{non-reconstructible entry that moved off-diagonal after pivoting}

\STATE $x_{k,\col}\leftarrow p^{-1}x_{k,\col}$; $x_{\ell,\col}\leftarrow x_{\ell,\col}-l_{\ell}x_{k,\col}$ for all $\ell \neq k$
\STATE $G_{k,\col}\leftarrow p^{-1}G_{k,\col}$; $G_{\ell,\col}\leftarrow G_{\ell,\col}-l_{\ell}G_{k,\col}$ for all $\ell \neq k$
\STATE $B_{\col,k}\leftarrow p^{-1}B_{\col,k}$; $B_{\col,\ell}\leftarrow B_{\col,\ell}-B_{\col,k}u_{\ell}$ for all $\ell \neq k$
\COMMENT{the update is performed in the ``extended matrix'' fashion for $B$ and $y$ too}
\STATE $y_{\col,k}\leftarrow p^{-1}y_{\col,k}$; $y_{\col,\ell}\leftarrow y_{\col,\ell}-y_{\col,k}u_{\ell}$ for all $\ell \neq k$

\STATE $d_\ell\leftarrow d_\ell-p^{-1}l_\ell u_\ell$ for all $\ell=k+1$ to $n$
\IF[$d_q$ may be reconstructible after the pivoting --- but we store it explicitly anyway]{$q\neq k$}
\STATE $d_q \leftarrow \frac{G_{q,\col}B_{\col,q}}{s_{\sigma(q)}-s_q}$
\ENDIF
\STATE $l_k \leftarrow -1;u_k \leftarrow -1;d_k \leftarrow 0$\COMMENT{prepares to overwrite $d(k)$ with the diagonal of the inverse}
\STATE $u_\ell \leftarrow 0$ for $\ell=k+1$ to $n$ \COMMENT{permutes the entries of $u$ to create the right matching for the update of the formula\eqref{permvers}. Notice that the variable $u$ holds now the entries of $(L^{-1}P^{-1})_{k,\ell}$ and $l$ holds  the entries of $(U^{-1})_{\ell,k}$}
\STATE $u_{\sigma(\ell)}\leftarrow u_{\ell}$ for all $\ell=1$ to $n$
\STATE $d_\ell \leftarrow d_\ell+p^{-1}l_\ell u_\ell$ for all $\ell=1$ to $k$
\ENDFOR
\COMMENT{continues in the next page}
\end{algorithmic}
\caption{Computing the representation of $T^{-1}$ (and solving systems with matrix $T$ and $T^*$)}\label{a:inv}
\end{algorithm}
\begin{algorithm}
\begin{algorithmic}
\STATE \COMMENT{continues: last step ($k=n$) of the algorithm}\\

\STATE $l_{\ell}\leftarrow \frac{G_{\ell,\col}B_{\col,n}}{s_\ell-s_n}$ for all $\ell=1$ to $n-1$
\STATE $u_{\ell}\leftarrow \frac{G_{n,\col}B_{\col,\ell}}{s_{\sigma(n)}-s_{\sigma(\ell)}}$ for all $\ell=1$ to $n-1$
\STATE $p\leftarrow d_n$
\STATE $x_{n,\col}\leftarrow p^{-1}x_{n,\col}$; $x_{\ell,\col}\leftarrow x_{\ell,\col}-l_{\ell}x_{n,\col}$ for all $\ell=1$ to $n-1$
\STATE $G_{n,\col}\leftarrow p^{-1}G_{n,\col}$; $G_{\ell,\col}\leftarrow G_{\ell,\col}-l_{\ell}G_{n,\col}$ for all $\ell=1$ to $n-1$
\STATE $B_{\col,n}\leftarrow p^{-1}B_{\col,n}$; $B_{\col,\ell}\leftarrow B_{\col,\ell}-B_{\col,n}u_{\ell}$ for all $\ell=1$ to $n-1$
\STATE $y_{\col,n}\leftarrow p^{-1}y_{\col,n}$; $y_{\col,\ell}\leftarrow y_{\col,\ell}-y_{\col,n}u_{\ell}$ for all $\ell=1$ to $n-1$

\STATE $l_n \leftarrow -1;u_n \leftarrow -1;d_n \leftarrow 0$
\STATE $u_{\sigma(\ell)}\leftarrow u_{\ell}$ for all $\ell=1$ to $n$
\STATE $d_\ell \leftarrow d_\ell+p^{-1}l_\ell u_\ell$ for all $\ell=1$ to $n$

\STATE $y_{\col,\sigma(\ell)}=y_{\col,\ell}$ for all $\ell=1$ to $n$
\COMMENT{undoes the pivoting on the rows of $y$ and $B$}
\STATE $B_{\col,\sigma(\ell)}=B_{\col,\ell}$ for all $\ell=1$ to $n$
\RETURN $G,B,s,d$ \COMMENT{generators of the inverse}
\RETURN $x,y$ \COMMENT{solutions of $Tx=b$ and $yT=c$}

\end{algorithmic}
\addtocounter{algorithm}{-1}
\caption{(continued) Computing the representation of $T^{-1}$ (and solving systems with matrix $T$ and $T^*$)}
\end{algorithm}
\paragraph{Comments} It is worth noting with the same run of the GKO algorithm we can compute $\diag(T^{-1})$ and solve linear systems with matrices $T$ and $T^*$, as the two algorithms share many of their computations. In particular, the solutions of the two systems giving $T^{-1}G$ and $BT^{-1}$, which are the generators of $T^{-1}$, are computed by the algorithm with no additional effort: the transformations on $G$ and $B$ needed to solve them are exactly the ones that are already performed by the factorization algorithm. Also observe that, since the computation of $(T^{-1})_{i,i}$ spans steps $i$ to $n$, while $d_i$ is needed from step 1 to step $i$, we may reuse the vector $d$ to store the diagonal of the inverse. The resulting algorithm has a total computational cost of $(8r+2m_1+2m_2+5)n^2$ ops, if we solve at the same time a system $Tx=b$ with $b \in \C^{n \times m_1}$ and a system $yT=c$ with $c \in \C^{m_2 \times n}$ (otherwise just set $m_1=0$ and/or $m_2=0$). The only extra storage space needed is that used for $u$, $l$ and $\sigma$, i.e., the space required to store $2n$ real numbers and $n$ integer indices.

Another observation is that we did not actually make use of the fact that the diagonal of $T$ is non-reconstructible: in principle, this approach works even if $C$ is a Cauchy matrix with respect to two different node vectors $t$ and $s$. This might be useful in cases in which we would rather not compute explicitly the diagonal elements, e.g. because $t_i-s_i$ is very small, and thus would lead to ill-conditioning.

\section{Numerical experiments}\label{s:experiments}
\paragraph{Speed measurements} 
The speed experiments were performed on a FORTRAN 90 implementation of the proposed algorithms. The compiler used was the \texttt{lf95} FORTRAN compiler version 6.20c, with command-line options \texttt{-o2 -tp4 -lblasmtp4}. The experiments took place on two different computers:
\begin{description}
 \item [C1] a machine equipped with four Intel\textregistered{} Xeon\texttrademark{} 2.80Ghz CPUs, each equipped with 512kb of L2 cache, and 6 GB of RAM. Since we did not develop a parallel implementation, only one of the processors was actually used for the computations.
 \item [C2] a machine equipped with one Intel\textregistered{} Pentium\textregistered{} 4 2.80Ghz CPU with 1024kb of L2 cache, and 512Mb of RAM.
\end{description}

As an indicative comparison with a completely different algorithm, with different stability characteristics, we also reported the computational time for the solution of a (different!) Toeplitz system of the same size with the classical TOMS729 routine from Chan and Hansen, which is a Levinson-based Toeplitz solver. For C1, we reported the times of the Matlab backslash solver as a further comparison.

The results are shown in \autoref{time}.
\begin{table}[ht]
C1
 \begin{tabular}{c|ccc|cc}
n & $O(n^2)$-space GKO & Extended matrix & Downdating & TOMS729 & Backslash\\
\hline\\
128 & 1.4868e-03 & 2.0160e-03 & 2.1330e-03 & 9.3602e-04 & 3.4408e-03\\
256 & 7.7040e-03 & 7.6757e-03 & 8.1744e-03 & 3.5895e-03 & 1.6893e-02\\
512 & 6.0557e-02 & 2.9159e-02 & 3.0104e-02 & 1.3818e-02 & 8.4146e-02\\
1024 & 2.4182e-01 & 1.1922e-01 & 1.2351e-01 & 5.4675e-02 & 4.4138e-01\\
2048 & 9.6240e-01 & 4.5570e-01 & 4.6561e-01 & 2.5636e-01 & 2.6087e+00\\
4096 & 4.6003e+00 & 1.8935e+00 & 1.9060e+00 & 1.0242e+00 & 1.6023e+01\\
8192 & 3.0398e+01 & 9.2616e+00 & 8.1778e+00 & 5.4107e+00 &   Out of memory\\
16384 &   Out of memory & 4.0561e+01 & 3.6369e+01 & 2.3309e+01 &   Out of memory\\
32768 &   Out of memory & 1.9054e+02 & 1.6358e+02 & 1.0447e+02 &   Out of memory\\
65536 &   Out of memory & 7.9284e+02 & 7.0361e+02 & 4.1717e+02 &   Out of memory\\
 \end{tabular}
\medskip

C2
 \begin{tabular}{c|ccc|c}
  n & $O(n^2)$-space GKO & Extended matrix & Downdating & TOMS729\\
\hline\\
128 & 1.5779e-03 & 2.1675e-03 & 2.2298e-03 & 4.4989e-04\\
256 & 6.1923e-03 & 8.1009e-03 & 8.1240e-03 & 1.5978e-03\\
512 & 6.3149e-02 & 3.1964e-02 & 3.1461e-02 & 5.9812e-03\\
1024 & 3.0802e-01 & 1.2740e-01 & 1.2386e-01 & 2.3439e-02\\
2048 & 1.2481e+00 & 5.1163e-01 & 4.9413e-01 & 9.7404e-02\\
4096 & 4.8710e+00 & 2.0433e+00 & 1.9761e+00 & 3.8726e-01\\
8192 &   Out of memory & 8.4085e+00 & 8.0465e+00 & 1.5995e+00\\
16384 &   Out of memory & 3.7214e+01 & 3.3758e+01 & 7.3925e+00\\
32768 &   Out of memory & 1.9180e+02 & 1.5883e+02 & 4.8315e+01\\
65536 &   Out of memory & 9.0789e+02 & 7.9906e+02 & 2.6792e+02\\
\end{tabular}

\caption{Speed experiments: CPU times in seconds for the solution of Cauchy-like/Toeplitz linear systems with the different algorithms}\label{time}
\end{table}

\paragraph{Comments} It is clear from the table that two different behaviors arise for different sizes of the input. For small values of $n$, the winner among the GKO variants is the traditional $O(n^2)$-space algorithm, due to its lower computational cost of $(4r+2m+1)n^2$ instead of $(6r+2m+\frac 32)n^2$ (for these tests, $r=2$, $m=1$). As the dimension of the problem increases, cache efficiency starts to matter, and the traditional algorithm becomes slower than its counterparts. This happens starting from $n\approx256-512$. Quick calculations show that the memory occupation of the full $n \times n$ matrix is 512kb for $n=256$ and 2Mb for $n=512$, so the transition takes indeed place when the $O(n^2)$ algorithm starts to suffer from cache misses. 

The three GKO variants are slower than TOMS729; this is also due to the fact that the implementation of the latter is more mature than the GKO solvers we developed for this test; it uses internally several low-level optimizations such as specialized BLAS routines with loop unrolling.

\paragraph{Accuracy measurements} For the accuracy experiments, we chose four test problems, the first two inspired from Boros, Kailath and Olshevsky \cite{bko}, the third taken from Gohberg, Kailath and Olshevsky \cite{gko95}, and the fourth from Sweet \cite{sweet} (with a slight modification).
\begin{description}
 \item [P1] is a Cauchy-like matrix with $r=2$, nodes $t_i=a+ib$, $s_j=jb$ for $a=1$ and $b=2$, and generators $G$ and $B$ such that $G_{i,1}=1$, $G_{i,2}=-1$, $B_{1,j}=(-1)^j$, $B_{2,j}=2$. It is an example of a well-conditioned Cauchy-like matrix; in fact, for $n=512$, its condition number (estimated with the \matlab{} function \lstinline{condest}) is 4E+02, and for $n=4096$ it is 1.3E+03.
 \item [P2] is the same matrix but with $a=1$ and $b=-0.3$. It is an ill-conditioned Cauchy like matrix; in fact, the condition number estimate is 1E+17 for $n=512$ and 5E+20 for $n=4096$.
\item [P3] is the Gaussian Toeplitz matrix \cite{gko95}, i.e., the Toeplitz matrix defined by $T_{i,j}=a^{(i-j)^2}$, with size $n=512$ and different choices of the parameter $a \in (0,1)$. It is an interesting test case, since it is a matrix for which the Levinson-based Toeplitz solvers are unstable \cite{gko95}. Its condition number estimate is 7E+09 for $a=.90$ and 3E+14 for $a=0.93$.
\item [P4] is a Cauchy-like matrix for which generator growth is expected to occur \cite{sweet}. We chose $n=128$, the same nodes as \textbf{P1}, and generators defined by $G=[a,\,a+\varepsilon f]$, $B=[a+\varepsilon g,\,-a]^T$, where $\varepsilon=10^{-12}$ and $a,f,g$ are three vectors with random entries uniformly distributed between 0 and 1 (generated with the Matlab \lstinline{rand} function). Notice that the absolute values of the entries of $GB$ is about 1e-12, and their relative accuracy is about 1e-04.
\end{description}

For \textbf{P1} and \textbf{P2}, we chose several different values of $n$, for each of them we computed the product $v=Ce$ (where $e=[1 1 \dots 1]^T$) with the corresponding matrix $C$, and applied the old and new GKO algorithms to solve the system $Cx=v$. We computed the relative error as
\begin{equation}\label{relativeerror}
 err=\frac{\norm{x-e}}{\norm{e}}.
\end{equation}
As a comparison, for \textbf{P2} we also reported the accuracy of Matlab's unstructured solver (backslash), which is an $O(n^3)$ algorithm based on Gaussian elimination.

For \textbf{P3}, we solved the problem $Tx=v$, with $T$ the Gaussian Toeplitz matrix and $v=Te$, for different values of the parameter $a$, with several different methods: reduction to Cauchy-like form followed by one of the three GKO-Cauchy solvers presented in this papers, Matlab's backslash, and the classical Levinson Toeplitz solver TOMS729 by Chan and Hansen \cite{toms729}. The errors reported in the table are computed using the formula \eqref{relativeerror}.

For \textbf{P4}, we generated five matrices, with the same size and parameters but different choices of the random vectors $a$ and $f$. We used the same right-hand side and error formula as in the experiments \textbf{P1} and \textbf{P2}. The condition number estimates of the matrices are reported as well.

The results are shown in \autoref{accuracy}.
\begin{table}[ht]
P1
 \begin{tabular}{c|ccc}
  $n$ & $O(n^2)$-space GKO & Extended matrix & Downdating\\
\hline
128 & 1.262497e-15 & 1.160020e-15 & 1.062489e-15\\
256 & 1.520695e-15 & 1.812447e-15 & 1.463218e-15\\
512 & 2.979162e-15 & 3.063677e-15 & 3.091645e-15\\
1024 & 2.790466e-15 & 3.429299e-15 & 3.068041e-15\\
2048 & 4.568803e-15 & 5.921849e-15 & 5.044874e-15\\
4096 & 5.231503e-15 & 7.448194e-15 & 5.461259e-15\\
8192 & 7.491095e-15 & 1.250913e-14 & 7.287788e-15\\
16384 & Out of memory & 1.648221e-14 & 1.154215e-14\\
32768 & Out of memory & 2.624266e-14 & 1.757211e-14\\
65536 & Out of memory & 3.929339e-14 & 2.209921e-14\\
 \end{tabular}

\medskip

P2
 \begin{tabular}{c|ccc|c}
  $n$ & $O(n^2)$-space GKO & Extended matrix & Downdating & Backslash\\
\hline
128  & 4.226744e-05 & 4.226745e-05 & 4.226745e-05 & 6.943135e-05\\
256  & 2.498321e-03 & 2.498321e-03 & 2.498321e-03 & 1.681902e-03\\
512  & 1.307574e-01 & 1.307574e-01 & 1.307574e-01 & 2.151257e-01\\
1024 & 1.634538e+01 & 1.634538e+01 & 1.634538e+01 & 1.872503e+01\\
2048 & 4.367616e+02 & 4.367616e+02 & 4.367616e+02 & 2.341069e+03\\
4096 & 2.311074e+04 & 2.311075e+04 & 2.311075e+04 & 1.124867e+03\\
 \end{tabular}

\medskip

P3
 \begin{tabular}{c|ccc|c|c}
  $a$ & $O(n^2)$-space GKO & Extended matrix & Downdating & Backslash & TOMS729\\
\hline
0.85 & 2.916298e-10 & 1.584459e-10 & 1.960486e-10 & 3.083869e-11 & 1.631254e-10\\
0.87 & 7.080698e-10 & 6.933175e-10 & 6.234554e-10 & 3.672145e-10 & 2.736550e-09\\
0.90 & 1.928754e-07 & 2.741270e-07 & 1.807345e-07 & 1.402849e-07 & 3.122989e-06\\
0.91 & 2.690618e-04 & 1.645149e-04 & 2.647343e-04 & 1.359575e-06 & 1.208559e-04\\
0.92 & 8.092059e-05 & 1.165346e-04 & 1.540948e-04 & 4.638024e-05 & 1.023501e-02\\
0.93 & 5.766805e-03 & 6.569097e-03 & 6.182359e-03 & 2.532194e-03 & 2.486232e+00\\
0.94 & 3-767035e-01 & 2.111118e+00 & 2.837602e-01 & 1.116684e+00 & 1.767069e+03\\
 \end{tabular}

\medskip

P4
\begin{tabular}{ccc|c}
  $O(n^2)$-space GKO & Extended matrix & Downdating & \texttt{condest(C)}\\
\hline
1.678593e-01 & 1.678593e-01 &      1.681295e-01 & 4e+04\\
1.170304e-01 &      1.160042e-01 &      1.125991e-01 & 4e+03\\
2.805922e+01 &     2.800674e+01 &      2.797860e+01 & 1e+05\\
5.552484e-02 &      5.173933e-02 &      5.568090e-02 & 1e+04\\
6.540661e-02 &      6.760469e-02 &      7.393447e-02 & 1e+03
\end{tabular}
 
\caption{Relative forward errors}\label{accuracy}
\end{table}
\paragraph{Comments}
There are no significant differences in the accuracy of the three variants of GKO. This shows that, at least in our examples, despite the larger number of operations needed, the space-efficient algorithms are as stable as the original GKO algorithm. On nearly all examples, the stability is on par with that of Matlab's backslash operator. When applied to critical Toeplitz problems, the GKO-based algorithms can achieve better stability results than the classical Levinson solver TOMS729.

In the generator growth case \textbf{P4}, the accuracy is very low, as expected from the theoretical bounds; nevertheless, there is no significant difference in the accuracy of the three versions.

We point out that a formal stability proof of the GKO algorithm cannot be established, since it is ultimately based on Gaussian elimination with partial pivoting, for which counterexamples to stability exist, and since in some limit cases there are other issues such as generators growth \cite{sweet}: i.e., the growth of the elements of $G$ and $B$ (but not of $L$ and $U$) along the algorithm. 
However, both computational practice and theoretical analysis suggest that the GKO algorithm is in practice a reliable algorithm \cite{olshevsky}. Several strategies, such as the one proposed by Gu \cite{gu95}, exist in order to avoid generator growth, and they can be applied to both the original GKO algorithm and its space-efficient versions. The modified versions of GKO are not exempt from this stability problem, since they perform the same operations as the original one plus some others; nevertheless, when performing the numerical experiments for the present paper, we encountered no case in which the $O(n)$-space algorithms suffer from generator growth while the original version does not.

\paragraph{A posteriori accuracy test}
We tested on the experiment \textbf{P2} the \emph{a posteriori} accuracy test  mentioned at the end of \autoref{s:dd}; i.e., solving the system with the downdating approach and then comparing the values of $B$ before and after the algorithm. In \autoref{t:apost}, we report the value of the relative error $\norm{B-B'}/\norm{B}$, where $B'$ is the value of the variable initially holding the second generator $B$ at the end of the algorithm. We compare it with the relative residual $\norm{C \tilde x-b}/\norm{b}$, where $C$ and $b$ are the system matrix and right-hand side of the experiment \textbf{P2}, and $\tilde x$ is the solution computed by the downdating algorithm.
\begin{table}P2
 \begin{tabular}{c|cc}
 $n$ & \emph{A posteriori} test & Relative residual\\
 \hline
 128  & 2.6678472e-12 & 2.0294501e-13\\
 256  & 5.6104849e-12 & 4.2803993e-13\\
 512  & 8.8302484e-12 & 1.6978596e-12\\
 1024 & 1.4490236e-10 & 2.9394916e-09\\
 2048 & 6.5423206e-10 & 4.9881189e-07\\
 4096 & 6.2470657e-10 & 3.5584063e-05\\
 \end{tabular}
 \caption{Accuracy of the \emph{a posteriori} accuracy test}\label{t:apost}
\end{table}
At least in this experiment, the proposed test is not able to capture the instability of the algorithm; the computation of the relative residual is more accurate as an \emph{a posteriori} test to estimate the accuracy of the solution.

\paragraph{Speed comparison with Matlab's backslash}
\begin{table}
P2
\begin{tabular}{c|ccc}
 $n$ & Downdating(Matlab) & Downdating(Fortran) & Backslash(assembling+solving) \\
\hline
128  & 4.831e-02 & 2.046e-03 & 1.446e-02 + 2.703e-03\\
256  & 9.224e-02 & 7.519e-03 & 1.945e-02 + 1.237e-02\\
512  & 2.125e-01 & 2.783e-02 & 4.016e-02 + 6.624e-02\\
1024 & 5.809e-01 & 1.114e-01 & 1.224e-01 + 3.527e-01\\
2048 & 1.849e+00 & 4.329e-01 & 4.518e-01 + 2.215e+00\\
4096 & 7.099e+00 & 1.732e+00 & 1.910e+00 + 1.426e+01\\
\end{tabular}
\caption{CPU times (in seconds) on the machine \textbf{C1} for the Matlab and Fortran implementation of downdating and for the Matlab backslash operator}\label{vsbackslash}
\end{table}
We have compared the speed of the Matlab and Fortran implementations of downdating GKO with the cost of assembling the full matrix $C$ in Matlab and solving the system with the backslash operator. The results are in \autoref{vsbackslash}. The experiments were performed on \textbf{C1}, and the version of Matlab used was 7 (R14) SP1.

The comparison is not meant to be fair: on one hand, we are testing a $O(n^2)$ and a $O(n^3)$ algorithm; on the other, we are comparing an interpreted program, a compiled program and a call to a native machine-code library within Matlab. Starting from $n=2048$, even the Matlab version of the downdating algorithm is faster than backslash: for large values of $n$, the overhead of processing $O(n)$ instructions with the Matlab interpreter is amortized.

\paragraph{Inversion of Trummer-like matrices} We shall now turn to testing the algorithm for the structured inversion of Trummer-like matrices proposed in \autoref{s:trummer}. We chose two experiments, one with well-conditioned matrices and one with ill-conditioned ones. Notice that not all possible choices of the generators $G$ and $B$ are admissible for a Trummer-like matrix, since the displacement equation implies $G_{i,\col}B_{\col,i}=0$ for all $i$.
\begin{description}
\item[T1] The $n\times n$ Trummer matrix $T$ with $T_{i,i}=1$, nodes defined by $s_i=i/n$ and generators defined by $G_{i,1}=i$, $G_{i,2}=-1$, $B_{1,i}=\cos(\pi i/n)$, $B_{i,2}=G_{i,1}B_{1,i}$ for all $i=1,\dots,n$.
\item[T2] The $512\times 512$ diagonal-plus-rank-1 matrix depending on a parameter $\varepsilon$ and defined by $T=(1+\varepsilon)I-uu^T$, with $u=v/\norm{v}$ and $v_i=i/n$ for all $i=1,\dots,n$. Its inverse can be computed explicitly as $(1+\varepsilon)^{-1}(I+\varepsilon^{-1}uu^T)$, and its condition number is $\varepsilon^{-1}+1$. A diagonal-plus-rank-1 matrix is Trummer-like with $r=2$ with respect to any set of nodes; in this experiment, we used the node vector defined by $s_i=a+ib$ for all $i=1,\dots,n$, with $a=1$, $b=-0.3$.
\end{description}
We computed the relative errors
\[
 E_1=\frac{\norm{d'-d''}}{\norm{d''}},\,E_2=\frac{\norm{G'-G''}}{\norm{G''}}+\frac{\norm{B'-B''}}{\norm{B''}},\,E_3=\frac{\norm{T'-T''}}{\norm{T''}},
\]
where $G',B',d',T'$ are the generators, diagonal and full inverse computed by \autoref{a:inv}, and $G'', B'', d'', T''$ are their reference values computed with Matlab's function \lstinline{inv} for \textbf{T1} and with the exact formula for the inverse for \textbf{T2}. The results are reported in \autoref{t:inv}.
\begin{table}
T1
 \begin{tabular}{c|ccc|c}
  $n$ & $E_1$ & $E_2$ & $E_3$ & \texttt{condest(T)}\\
  \hline
128 &   5.4547208e-16&   1.1319021e-14&   2.9390472e-15&   5.8720426e+02\\256 &   7.0728226e-16&   2.8467666e-14&   7.4006119e-15&   1.1929568e+03\\512 &   9.3436757e-16&   4.1000919e-14&   1.1661137e-14&   2.4043890e+03\\1024 &   1.3348516e-15&   1.2770834e-13&   2.9260324e-14&   4.8272132e+03\\2048 &   1.8731364e-15&   2.3347806e-13&   5.6916515e-14&   9.6728404e+03\\4096 &   2.7481278e-15&   2.5345941e-13&   7.9356066e-14&   1.9364084e+04\\  
 \end{tabular}
 T2
 \begin{tabular}{c|ccc}
  $\varepsilon$ & $E_1$ & $E_2$ & $E_3$\\
  \hline
  1e-03 &
   2.2655145e-11
&
   5.9001177e-11
&
   3.0152973e-11
\\
1e-06 &
   4.0447578e-08
&
   8.0919137e-08
&
   4.1084327e-08
\\
1e-09 &
   4.0899169e-05
&
   8.1796690e-05
&
   4.1263900e-05
\\
1e-12 &
   3.2571481e-02
&
   6.6239581e-02
&
   3.2914231e-02
\\
1e-15 &
   1.4160667e+00
&
   4.8195274e+00
&
   1.7288419e+00
\\
 \end{tabular}
 \caption{Accuracy of the algorithm for inverting Trummer-like matrices}\label{t:inv}
\end{table}
In both cases, the algorithm is able to reach good accuracy, compatibly with the restrictions imposed by the condition number of the matrices.

\paragraph{Code availability} Fortran and \matlab{} implementations of the algorithms presented here are available online on \url{http://arxiv.org/e-print/0903.4569} along with the e-print of this paper.

\section{Conclusions}\label{s:conclusions}
In this paper, we proposed a new $O(n)$-space version of the GKO algorithm for the solution of Cauchy-like linear systems. Despite the slightly larger number of operations needed, this algorithm succeeds in making a better use of the internal cache memory of the processor, thus providing an improvement with respect to both the customary GKO algorithm and a similar $O(n)$-space algorithm proposed by Rodriguez \cite{rod}, \cite{ar}. Starting from $n \approx 500-1000$, the algorithm outperforms these two versions of GKO. When applying this algorithm to the special case of inversion of Trummer-like matrices, several small optimizations reduce the total number of operations needed.

\paragraph{Acknowledgements} The author would like to thank Antonio Aric\`o, Dario Bini and Vadim Olshevsky for helping him with many useful discussions on the subject of this paper. The anonymous referees suggested some additional numerical experiments and made several suggestions which greatly helped in improving the presentation of the  paper.

\bibliographystyle{abbrv}
\bibliography{cauchy}

\end{document}